\documentclass[12pt,a4paper]{amsart}
\usepackage{graphicx}
\usepackage{amsfonts}
\usepackage{amssymb}
\usepackage{amsthm}
\usepackage[centertags]{amsmath}
\usepackage{newlfont}
\usepackage{color}
\usepackage{graphicx,xcolor}
\usepackage{amsmath, amssymb, graphics}

\def\r{\mathbb R}
\def\e{\mathbb E}
\newtheorem{theorem}{Theorem}[section]
\newtheorem{corollary}[theorem]{Corollary}

\begin{document}
\title{Ruled translating solitons in Minkowski $3$-space}
\author{Muhittin Evren Aydin}
\email{meaydin@firat.edu.tr}
\address{Department of Mathematics\\
Faculty of Science, Firat University\\
Elazig, 23200, Turkey}
\author{Rafael L\'opez}
\email{rcamino@ugr.es}
\address{Departamento de Geometr\'{\i}a y Topolog\'{\i}a\\
Universidad de Granada\\
18071 Granada, Spain}
\date{}

\begin{abstract}
We characterize all ruled translating solitons in Minkowski $3$-space. In contrast to the Euclidean space, we find ruled translating solitons that are not cylindrical. These surfaces appear when the vector field that defines the rulings, viewed as a curve, is a lightlike straight line. We also classify all cylindrical translating solitons, obtaining surfaces that can be considered as analogous to the grim reapers of Euclidean space, but also other surfaces which have no a counterpart in the Euclidean space.
\end{abstract}

\maketitle

%%%%%%%%%%%%%%%%%%%%

%%%%%%%%%%%%%%%%%%%%%%%%%%%%

\section{Introduction}

\label{1} %%%%%%%%%%%%%%%%%%%%%%%%%%%%

Let $\mathbb{E}_{1}^{3}$ be the Minkowski $3$-space, that is, the affine
space $\mathbb{R}^{3}$ with canonical coordinates $(x,y,z)$ and endowed with
the indefinite metric $\langle ,\rangle =dx^{2}+dy^{2}-dz^{2}$. As usual, we
use the terminology spacelike, timelike and lightlike of the distinct types of vectors in $\mathbb{E}_{1}^{3}$ according its causality. Likewise as in Euclidean space  (\cite{hsi1,il}), there is a theory of the mean curvature flow in Minkowski space that can be dated back to works by Ecker, Huisken and Gerhardt (\cite{ec1,ec2,ec3,ge}). In general, the submanifolds to be considered are spacelike and although there are similarities with the Euclidean case,  there are also differences. For example, if the value of the codimension of the submanifold is important in the Euclidean space, in Minkowski space the spacelike condition of the submanifold is preserved by the mean curvature flow regardless of its codimension (\cite{ll}). 

In the theory of mean curvature flow, translating solitons play a remarkable role. A \textit{translating soliton} in $\mathbb{E}_{1}^{3}$ with respect to a vector $\vec{v}$, called the {\it velocity} of the flow,   is a non-degenerate surface whose mean curvature $H$
satisfies 
\begin{equation}
2H(p)=\langle N(p),\vec{v}\rangle ,  \label{eq1}
\end{equation}
for all $p\in M$, where $N$ is the unit normal vector field on $M$. A translating soliton  is a solution of the flow when $M$ evolves purely by translations
along the direction $\vec{v}$. In particular, $M+t\vec{v}$, 
$t\in \mathbb{R}$, satisfies that fixed $t$, the normal component of the
velocity vector $\vec{v}$ at each point is equal to the mean curvature at
that point.

In order to find  examples of translating solitons in Minkowski space, we will assume that the translating soliton  is  a {\it ruled surface}. A ruled surface in $\r^3$ is swept out by a straight line  moving along a fix curve $\gamma=\gamma(s)$ called the base of the surface. So, in this paper we investigate in the Lorentzian ambient space the following

\begin{quote}
\textbf{Problem.} Classify all ruled translating solitons in Minkowski $3$-space.
\end{quote}

 Let us notice that the definition of ruled surface is affine and not metric. A ruled surface admits a parametrization $X(s,t)=\gamma(s)+t w(s)$ where $\gamma(s)$ is a regular
curve and $w(s)$ is a nowhere vanishing vector field along $\gamma$. It is when imposing the translating soliton equation \eqref{eq1} that we use the differential-geometric concepts of the Minkowski space $\e_1^3$ or the Euclidean space $\e^3$. 

First we recall what happens with the ruled translating solitons in Euclidean space $\e^3$. A result of Hieu and Hoang proves that a ruled
translating soliton  is a plane or the rulings must be
parallel, that is, a cylindrical surface (\cite{Hieu}). Cylindrical translating solitons in $\e^3$ are called \textit{grim reapers} and can be
explicitly described. Indeed, after a dilation and a rigid motion of $\mathbb{E}^3$, we can
suppose that $\vec{v}=(0,0,1)$. Consider a cylindrical translating surface $X(s,t)=\gamma(s)+tw$ with respect to $\vec{v}$ whose rulings are all parallel to the vector $w\not=0$ and  $\gamma=\gamma(s)$ is a  curve contained
in an orthogonal plane to  $w$. After a rotation about the $z$-axis, let $w=\cos\theta e_1+\sin\theta e_3$, where $\{e_1,e_2,e_3\}$ is the standard
basis of $\mathbb{R}^3$ and  $e_3=\vec{v}$. If $\cos\theta=0$, it is immediate that the surface is a plane parallel to $\vec{v}$. If $\cos\theta\not=0$, the curve $\gamma$   is included in the plane spanned by $\{e_2,e\}$, where $e=-\sin\theta
e_1+\cos\theta e_3$. If $\gamma$ writes as $\gamma(s)=se_2+u(s)e$, then  the function $u$ satisfies the ODE $u''=\cos\theta(1+u^{\prime 2})$. The solution of this equation, the so-called grim reaper, is
$$u(s)=-\frac{1}{ \cos\theta }\log(\cos(a+ s \cos\theta))+b,\quad a,b\in\r.$$
 In case that the rulings are orthogonal to the velocity $\vec{v}$, then the grim reaper is $u(s)=-\log(\cos (s))$ up to suitable constant and translation of $s$.

This paper is motivated by the fact that, in principle, there are more ruled translating solitons in Minkowski space than in Euclidean space. First because there are two types of non-degenerate surfaces, namely, spacelike surfaces and timelike surfaces. The notion of the mean curvature coincide in both cases, but for timelike surfaces the Weingarten map is not necessarily diagonalizable, so it is possible that do not exist principal curvatures. The second difference with the Euclidean case is   the causal character of the  vector $\vec{v}$ in \eqref{eq1}. If in the Euclidean context, the velocity $\vec{v}$ can be prescribed after a rigid motion, in $\e_1^3$ the vector $\vec{v}$ can be also prefixed but now we must distinguish three cases depending if  $\vec{v}$ is spacelike, timelike or lightlike. Definitively, the problem of classification of ruled translating solitons in $\e_1^3$ is richer than in $\e^3$.

As a consequence of our study,  we highlight two important results which are not to be derived in Euclidean space. The first is about cylindrical translating solitons:
\begin{quote} {\it When the velocity $\vec{v}$ is lightlike, any cylindrical surface whose rulings are parallel to $\vec{v}$ is a translating soliton.}
\end{quote}
The second result asserts that surprisingly, as we shall demonstrate in this article,   
\begin{quote} {\it There are ruled translating solitons in $\e_1^3$  that are not cylindrical. }
\end{quote}

The structure of the paper is the following. After some preliminaries on local classical differential geometry of surfaces in $\e_1^3$ (Section \ref{sec2}), we classify in Section \ref{sec3} the ruled translating solitons when the rulings are parallel to a fix direction (cylindrical surfaces). We will obtain in Theorem \ref{t1} a complete description of such surfaces. In the case that the rulings are not lightlike, these surfaces can be considered as the analogous of the grim reapers of  in Euclidean space (Corollary \ref{c1}), but we will obtain more. Furthermore, in the case that the rulings are lightlike,    the  rulings must be parallel to $\vec{v}$ but with the freedom that the base curve of the surface is arbitrary.   Recently the authors have proved that the  cylindrical translating solitons in Minkowski space are the only translating solitons obtained by the technique of separation of variables (\cite{al}).

In   Section \ref{sec4}, we address the question on the existence of non-cylindrical ruled surfaces that are translating solitons. Recall that in Euclidean space, the only possibility is that the surface is a plane (\cite{Hieu}). As a conclusion of Theorems \ref{t2}, \ref{t3} and \ref{t4}, we will prove that there are  different examples than planes. These surfaces only appear when the direction of the rulings, $s\mapsto w(s)$, viewed as a curve of $\e_1^3$, is a lightlike straight line.  We will show the explicit parametrizations of these surfaces at the end of the section. As an illustrative example of this type of surfaces, let $\vec{v}=(1,0,0)$. The  surfaces 
$$X(s,t)=\left(\log(s),\frac{1}{2s},-\frac{1}{2s}\right)+t(1,s,s),\quad s>0, t>\frac12$$
and 
$$Y(s,t)=\left(-\frac12\log(1+s^2),\arctan(s)+s,s\right)+t (1,s,s),\quad s\in\r, t>-\frac32.$$
are translating solitons with respect to $\vec{v}$ and both are not cylindrical.

%%%%%%
\section{Preliminaries}\label{sec2}

We recall the notion of the mean curvature of a non-degenerate surface  in Minkowski space as well as a local expression of the Equation \eqref{eq1} for parametric surfaces. Much of the local surface theory in Minkowski space  is similar to the Euclidean space. Here we refer to \cite{lo,we} for details. The Lorentzian cross product of two vectors $\vec{a},\vec{b}\in\e_1^3$ is defined as the unique vector $\vec{a}\times \vec{b}$ such that $\langle \vec{a}\times \vec{b},\vec{c}\rangle=(\vec{a},\vec{b},\vec{c})$ for every $\vec{c}\in\e_1^3$, where $(\vec{a},\vec{b},\vec{c})$ denotes the determinant of the $3\times 3$ matrix formed by the vectors $\vec{a}$, $\vec{b}$ and $\vec{c}$.  According to the induced metric, a non-degenerate surface $M$ of $\e_1^3$ is spacelike (resp. timelike) if the metric is Riemannian (resp. Lorentzian). In such a case, one can define a unit normal vector field $N$ on $M$ which is timelike (resp. spacelike) if $M$ is spacelike (resp. timelike). Let $\epsilon=\langle N,N\rangle$ and denote $\nabla^0$ and $\nabla$ the   Levi-Civita connection of $\e_1^3$ and of $M$, respectively. For two tangent vector fields $U$ and $V$ on $M$, the Gauss formula is $\nabla_U^0 V=\nabla_U V+\sigma(U,V)$, where $\sigma$ is the second fundamental form of the immersion. Since $\sigma(U,V)$ is proportional to $N$, we have $\sigma(U,V)=\epsilon\langle\sigma(U,V),N\rangle N$. The mean curvature vector $\vec{H}$ is defined by $\vec{H}=\frac12\mbox{trace}(\sigma)$ and the (scalar) mean curvature $H$  by the relation $\vec{H}=HN$. Consequently,
$$H=\epsilon \langle \vec{H},N\rangle.$$
The Weingarten map $A_p$,  $p\in M$, is a self-adjoint endomorphism in the tangent plane $T_pM$ defined by the relation $\langle A_p U_p,V_p\rangle=\langle\sigma_p(U_p,V_p),N(p)\rangle$. Hence we obtain the formula $H=\dfrac{\epsilon}{2}\mbox{trace}(A)$ and similarly for the Gauss curvature $K=\epsilon\, \mbox{det}(A)$. In case that $M$ is spacelike  ($\epsilon=-1$), the Weingarten map is diagonalizable defining the principal curvatures as the eigenvalues of $A$. 

In order to compute the mean curvature $H$ in  \eqref{eq1}, we now obtain the expression of $H$ in local coordinates. Let $X=X(s,t)$ be a local parametrization of $M$. Choose $N$ as
$$N=\frac{X_s\times X_t}{|X_s\times X_t|}.$$
Then the mean curvature $H$ is  
$$
2H=-\frac{E\,(X_s,X_t,X_{tt})-2F\,(X_s,X_t,X_{st})+G\,(X_s,X_t,X_{ss})}{|EG-F^2|^{3/2}}.$$
Here $E$, $F$ and $G$ are the coefficients of the first fundamental form, with $-\epsilon(EG-F^2)=|X_s\times X_t|^2$. If $H_1=H_1(s,t)$ denotes the numerator of the right hand side in the above expression, then the translating soliton equation \eqref{eq1} is
\begin{equation}\label{eq2}
H_1=\epsilon(EG-F^2)\,(X_s,X_t,\vec{v}).
\end{equation}

To conclude this section, let point out the behaviour of the translating soliton equation \eqref{eq1} by dilations of the space. Let $M$ be a translating soliton with respect to $\vec{v}$. If $\lambda>0$ is a positive real number, then the dilation $\lambda M$ of $M$ has mean curvature $H/\lambda$. Since the Gauss map coincide at corresponding points of $M$ and $\lambda M$, then $\lambda M$ is a translating soliton with respect to the velocity $ \vec{v}/\lambda$. Or in other words, if we replace $\vec{v}$ by a multiple  $\lambda \vec{v}$, then a translating soliton with respect to $\vec{v}$ is a dilation of a translating soliton with respect to $\lambda \vec{v}$. On the other hand, a dilation of a ruled surface is also a ruled surface with rulings parallel to that of the initial surface. Thus, in our study of ruled translating solitons, the velocity $\vec{v}$ can be multiplied by a constant being the dilation of the ruled surface another ruled translating soliton with parallel rulings at corresponding points.

%%%%%%%%%%%%%%%%%%%%%%%%%%%%%%
\section{Cylindrical translating solitons}\label{sec3}
%%%%%%%%%%%%%%%%%%%%%%%%%%%%%%%%%%%%%%%%%%%%%%%

Consider a cylindrical surface in $\e_1^3$ and denote by  $w\not=0$   the direction of its rulings, with $|w|=1$ in case that $w$ is not lightlike. The parametrization of the surface is $X(s,t)=\gamma(s)+t w$, $t\in\r$, $s\in I\subset\r$, where $\gamma$ is a curve contained in an orthogonal plane to $w$. After a rigid motion of $\e_1^3$ we can assume that $w$ is $(1,0,0)$, $(0,0,1)$ or $(1,0,1)$. The first result is the classification of the cylindrical translating solitons.

\begin{theorem}\label{t1} Let $\vec{v}=(v_1,v_2,v_3)$ and let $M$ be a cylindrical surface. If $M$ is a translating soliton  with respect to $\vec{v}$, then $M$ is a plane parallel to $\vec{v}$ or $M$ must be one of the next three cases.
\begin{enumerate}
\item Spacelike rulings. Let $w=\left( 1,0,0\right) $. Then $M$ parametrizes as $X\left(
s,t\right) =\left( 0,s,u(s)\right) +tw$ where  
\begin{equation}\label{eq31}
-u''=\left\{ 
\begin{array}{ll}
-\left( 1-u^{\prime 2}\right) \left( v_{2}u'-v_{3}\right),& \mbox{if } 1-u^{\prime 2}>0 \\ 
\left( 1-u^{\prime 2}\right) \left( v_{2}u'-v_{3}\right),&  \mbox{if } 1-u^{\prime 2}<0.
\end{array}
\right.
\end{equation}
\item Timelike rulings. Let $w=\left( 0,0,1\right) $. Then  $M$ parametrizes as $X\left(
s,t\right) =\left( s,u(s) ,0\right) +tw$ where  
\begin{equation}\label{eq32}
u''=\left( 1+u^{\prime 2}\right) \left( v_2-v_{1}u^{\prime
}\right).
\end{equation} 
\item Lightlike rulings. Let $w=\left( 1,0,1\right) $. Then either  $M$ is a plane parallel to $\vec{v}$ or  $\vec{v}$ is parallel to the rulings and the base curve is arbitrary.
\end{enumerate}

\end{theorem}
\begin{proof}
\begin{enumerate}
\item Case $w=(1,0,0)$. Then the base curve $\gamma$ is contained in  the $yz$-plane. Since we want to consider the situation that $\gamma$ is a graph on the $y$-line, first we study the case that $\gamma$ is a vertical straight line $\gamma(s)=(0,c,s)$, $c\in\r$. Then the Equation \eqref{eq2} is simply $v_2=0$. Thus $\vec{v}=(v_1,0,v_3)$ and $M$ is a plane parallel to $\vec{v}$. In case that $\gamma$ is not a vertical line, we can locally write the base curve as $\gamma(s)=(0,s,u(s))$. Then $M$ is spacelike if $1-u'^2>0$ and  timelike if $1-u'^2<0$. The computation of \eqref{eq2} gives \eqref{eq31}.
\item Case $w=(0,0,1)$. Now the  curve $\gamma$ is contained in  the $xy$-plane. A first case to distinguish is that $\gamma$ is a line of type $\gamma(s)=(c,s,0)$, $c\in\r$. Now \eqref{eq2} is  $v_1=0$. Thus $\vec{v}=(0,v_2,v_3)$ and $M$ is a plane parallel to $\vec{v}$. In case that $\gamma$ is not a horizontal line, we can locally write $\gamma$ as $\gamma(s)=(s,u(s),0)$. Then $M$ is timelike because $EG-F^2=-1-u'^2$ and  Equation   \eqref{eq2} is \eqref{eq32}

\item Case $w=(1,0,1)$. Then $H=0$ and the translating soliton equation is simply $\langle N,\vec{v}\rangle=0$. The base curve is contained in the plane spanned by $(0,1,0)$ and $(1,0,-1)$. A first case is when $\gamma$ is the straight line $\gamma(s)=(s,c,-s)$, $c\in\r$. Then   Equation \eqref{eq2} is $v_2=0$  and $M$ is a plane parallel to $\vec{v}$. Otherwise, $\gamma$ writes as $\gamma(s)=(u(s),s,-u(s))$. The non-degeneracy condition of the surface is equivalent to   $u'\neq 0$ and \eqref{eq2}
is
\begin{equation*}
2v_{2}u'-v_1+v_{3}=0.
\end{equation*}
If $v_2=0$, then $v_1=v_3$ and $u(s)$ is an arbitrary function. In this case, $M$ is a translating soliton with respect to $\vec{v}=v_1(1,0,1)$, being $\vec{v}$ parallel to the rulings. If $v_2\not=0$, then $u(s)=(v_1-v_3)/(2v_2)s+a$, $a\in\r$. Now $\gamma$ is a straight line and $M$ is a plane parallel to $\vec{v}$ again. 

\end{enumerate}
\end{proof}

As a consequence of this theorem,   the family of cylindrical translating solitons in $\e_1^3$ where the rulings are lightlike has not a counterpart in the Euclidean space being now arbitrary  the base curve. In particular,    there are many (non-planar) translating solitons because the only condition is that the lightlike rulings must be parallel to the velocity $\vec{v}$.   It is important to point out that a ruled surface in $\e_1^3$ whose rulings are lightlike are   known  as \textit{null scrolls} in the literature and they have
an important role in Einstein's theory of relativity and physics of
gravitation in description of lightlike particles (\cite{bf}). Moreover, from the geometric viewpoint, these surfaces satisfy the equation $H^2=K$ and in the present case where the ruled surface is cylindrical, we have $H=K=0$.

In the final part of this section we will focus in the case that the rulings are not lightlike, obtaining  explicit parametrizations of the base curve.  Here we will have in mind that in Euclidean space, all grim reapers are  produced translating, scaling and rotating the standard grim reaper $u(s)=-\log(\cos(s))$.

\begin{enumerate}
\item  Case  $w=\left( 1,0,0\right) $. We consider two particular choices for the velocity $\vec{v}$.
\begin{enumerate}
\item Case  $\vec{v}=\left( v_{1},0,1\right) $. The solution of \eqref{eq31} depends if $u'^2<1$ or if $u'^2>1$. The integration yields 
\begin{equation}\label{gr1}
u(s) =\left\{ 
\begin{array}{ll}
-\log \cosh \left( s+a\right) +b& \mbox{or} \\ 
\log \sinh \left( s+a\right) +b,&
\end{array} 
\right.
\end{equation} 
respectively, where $a,b\in\r$. Both curves  appeared in \cite{I}.  Let us observe that the velocity $\vec{v}$ and the rulings are not necessarily orthogonal. For example, both are orthogonal if  $\vec{v}=(0,0,1)$. Let us observe that the same surfaces are translating solitons for lightlike vectors (if $v_1=\pm 1$) and timelike vectors (if $v_1^2<1$).
\item Case  $\vec{v}=(v_1,1,0)$. Then the solution of \eqref{eq31} is 
\begin{equation}\label{gr2}
u(s) =\left\{ 
\begin{array}{ll}
\pm \log (e^s+\sqrt{e^{2s}+a})+b,& a,b\in\r, a\not=0, \mbox{or} \\ 
\pm \mbox{ arctanh}(\sqrt{1-a e^{2s}})+b,& a,b\in\r, a>0. 
\end{array} 
\right.
\end{equation} 
\end{enumerate}

Assume now that $\vec{v}$ is not parallel to $w$ neither $(0,v_2,v_3)$ is lightlike. Then $v_2^2-v_3^2\not=0$.   Consider the rotations   $R_\varphi$ that pointwise fix the direction $w$ of the  rulings, which can be expressed by
$$R_\varphi=\left(\begin{array}{lll}1&0&0\\ 0&\cosh\varphi&\sinh\varphi  \\ 0&\sinh\varphi&\cosh\varphi\end{array}\right),$$
where $\varphi\in\r$. 
\begin{enumerate}
\item   If $v_2^2<v_3^2$, by choosing $\varphi$ such that $\tanh\varphi=-v_2/v_3$, then $R_\varphi(\vec{v})=(v_1,0,\tilde{v_3})$, $\tilde{v_3}\not=0$. After a dilation, the velocity $R_\varphi(\vec{v})$ can be assumed to be $(v_1,0,1)$. Thus the surface is a rotation and  a dilation of the surfaces given in \eqref{gr1}.
\item   If $v_2^2>v_3^2$, then there is $\varphi$ such that $\tanh\varphi=-v_3/v_2$. Then after the rotation $R_\varphi$ and a dilation of the space, we can assume that $\vec{v}=(v_1,1,0)$. Now the surfaces are rotations and dilations of the translating solitons given in \eqref{gr2}.
\end{enumerate}

\item Case  $w=\left( 0,0,1\right) $.  For the particular vector    $\vec{v}=\left( 0,1,v_{3}\right) $ in \eqref{eq32},  the equation to solve is $u''=1+u'^2$ and its solution is  \begin{equation}\label{gr3}
u\left(
s\right) =-\log (\cos \left( s+a\right)) +b, \quad a,b\in\r,
\end{equation}
 which coincides with the Euclidean grim reaper. Suppose now that $\vec{v}$ is not parallel to $w$ that is, $v_1\not=0$ or $v_2\not=0$ and consider the rotations 
$$R_\theta=\left(\begin{array}{lll}\cos\theta&-\sin\theta&0\\  \sin\theta&\cos\theta &0\\ 0&0&1\end{array}\right)$$
about the $z$-axis, $\theta\in\r$. After a dilation of $\e_1^3$, we can assume that $v_1^2+v_2^2=1$. Then there is  $\theta\in\r$ such that $R_\theta\vec{v}=(0,1,v_3)$. This rotation fixes the direction of the rulings, so the surface is again a ruled translating  surface with rulings parallel to $w$ but now the velocity $\vec{v}$ is $(0,1,v_3)$. By the previous work, we know that in such a case,  the surface is the Euclidean grim reaper \eqref{gr3}.
\end{enumerate}
  
  We summarize the above results as follows.

  \begin{corollary}\label{c1} Let $M$ be a cylindrical surface in $\e_1^3$ whose rulings are not lightlike. Let $\vec{v}$ be a vector that is not parallel to the rulings. If $M$ is a translating soliton with respect to $\vec{v}$, then $M$ is, up to a dilation and a rotation, 
   \begin{enumerate}
 \item  one of the surfaces   of equations \eqref{gr1} or \eqref{gr2} if the rulings are spacelike and the projection $(0,v_2,v_3)$ of $\vec{v}$ onto the $yz$-plane  is not lightlike; or
 \item the Euclidean grim reaper \eqref{gr3} if the rulings are timelike.  
 \end{enumerate}
 \end{corollary}

By analogy with the Euclidean case, we will call   \textit{Lorentzian grim reapers} the translating solitons of Corollary \ref{c1}. See Figure \ref{fig1}. 
\begin{figure}[hbtp]
\begin{center}
\includegraphics[width=.25\textwidth]{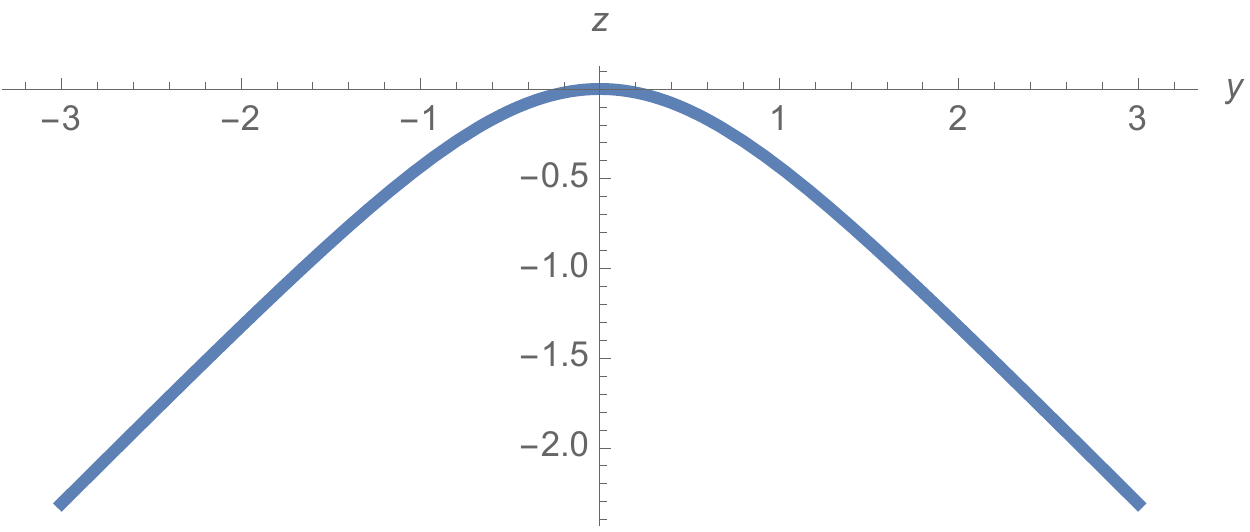}\includegraphics[width=.2\textwidth]{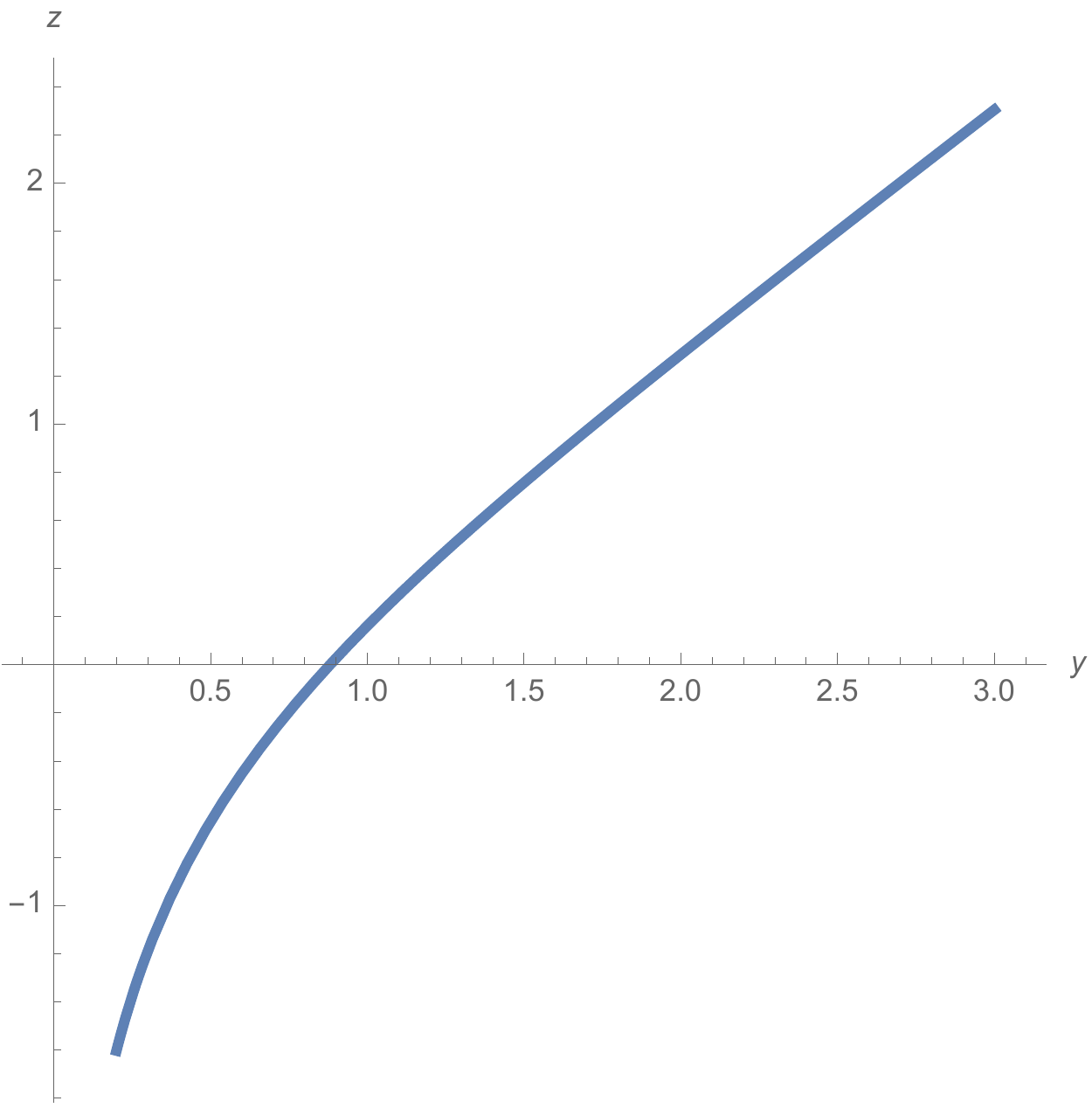}\includegraphics[width=.25\textwidth]{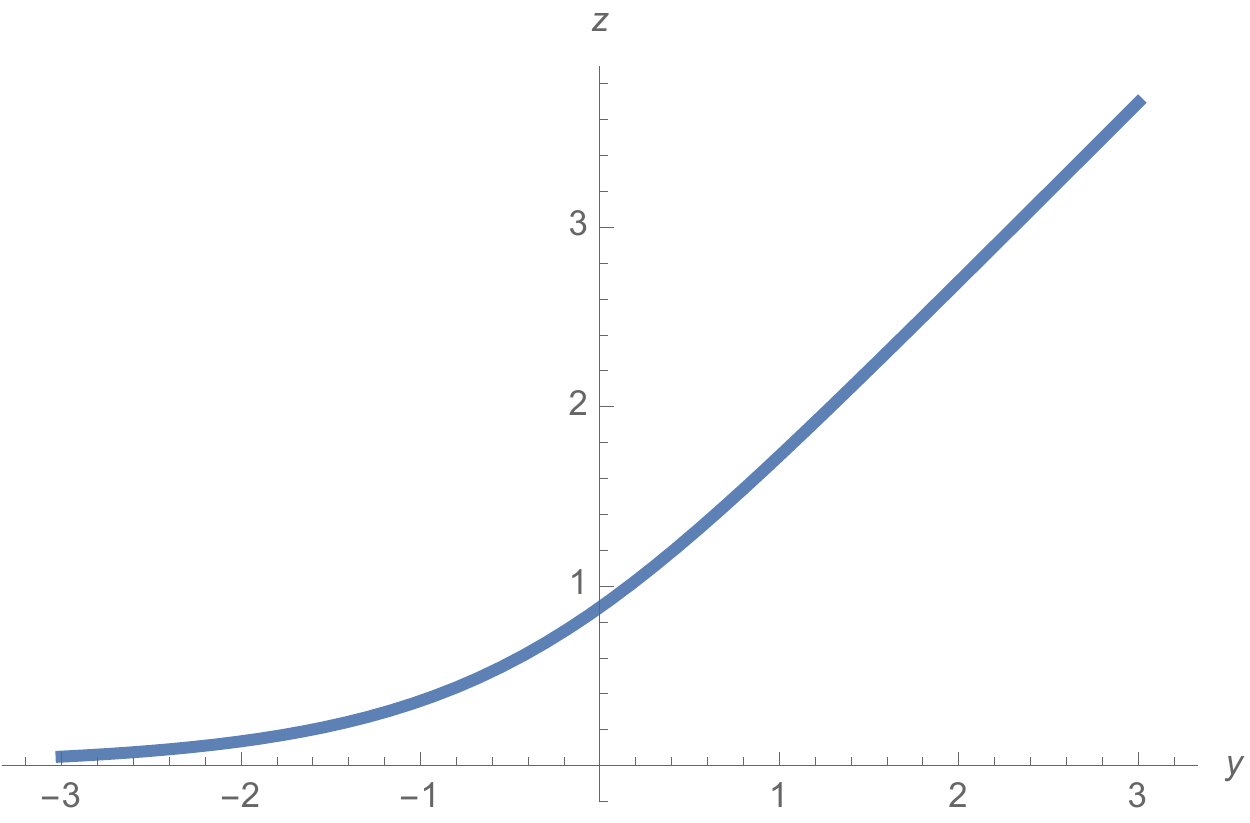}\includegraphics[width=.15\textwidth]{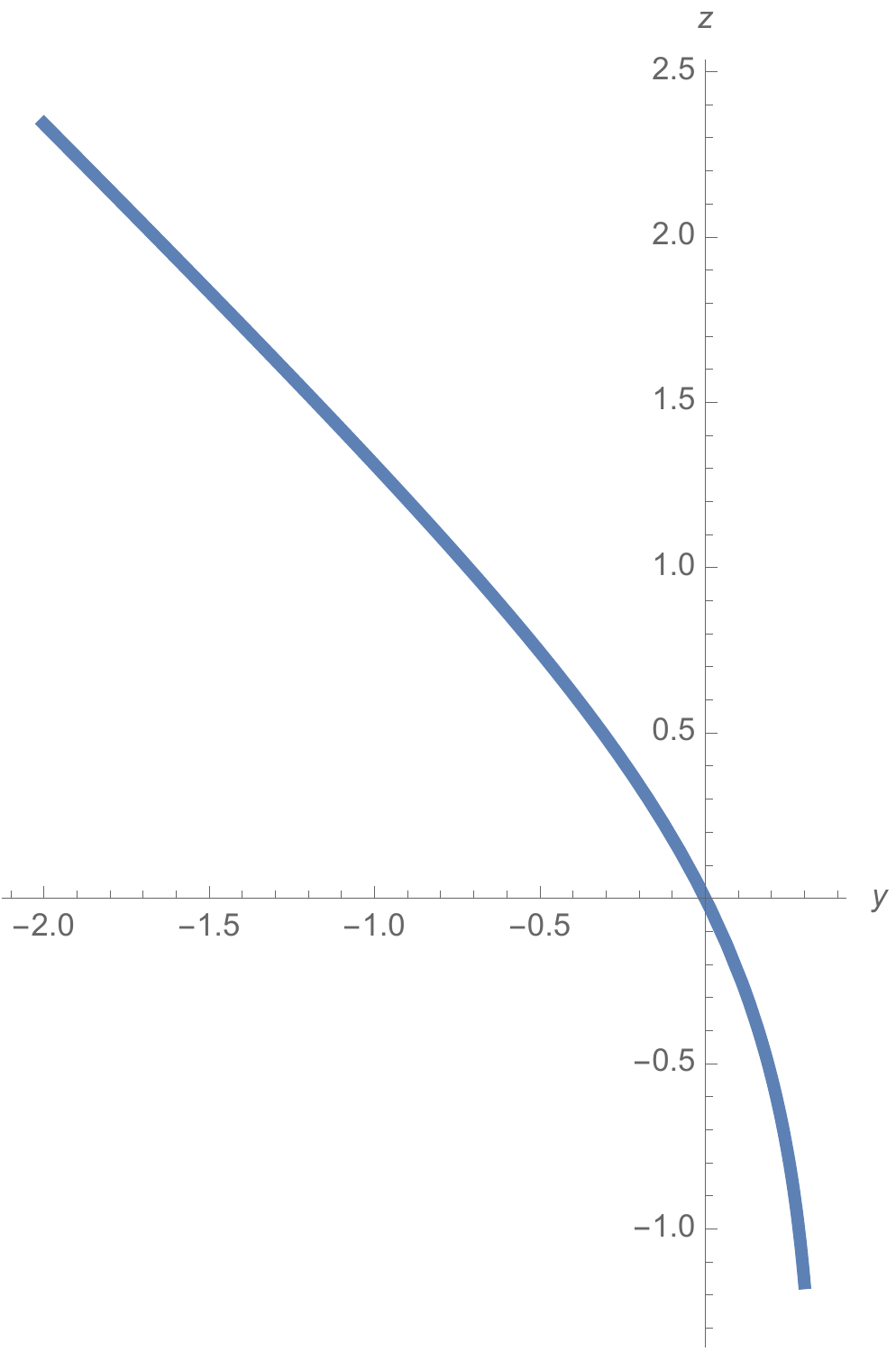}
\end{center}
\caption{Lorentzian grim reapers. From left to right:   $u(s)=-\log \cosh{s}$, $u(s)=\log \sinh{s}$, $u(s)=  \log (e^s+\sqrt{e^{2s}+1})$ and $u(s)= \mbox{ arctanh}(\sqrt{1- e^{2s}})$.}\label{fig1}
\end{figure}

Let us observe that there are cylindrical translating solitons with non-degenerate rulings that do not appear in this corollary. Indeed, consider the rulings given by   $w=(1,0,0)$. Then the case that   $\vec{v}$ satisfies $v_2^2-v_3^2=0$ does not enter in Corollary \ref{c1}.  After a dilation, let  $\vec{v}=(v_1,1,1)$. Then \eqref{eq31} is 
\begin{equation}\label{gr0}
u''=\left\{ 
\begin{array}{ll}
\left( 1-u^{\prime 2}\right)(1-u'),& \mbox{if }1-u^{\prime 2}>0 \\ 
-\left( 1-u^{\prime 2}\right)(1-u'),&\mbox{if }  1-u^{\prime 2}<0. 
\end{array} 
\right.
\end{equation} 
These equations are not integrable by quadratures. See Figure \ref{fig2} for numerical computations of solutions of \eqref{gr0}.

\begin{figure}[hbtp]
\begin{center}
\includegraphics[width=.4\textwidth]{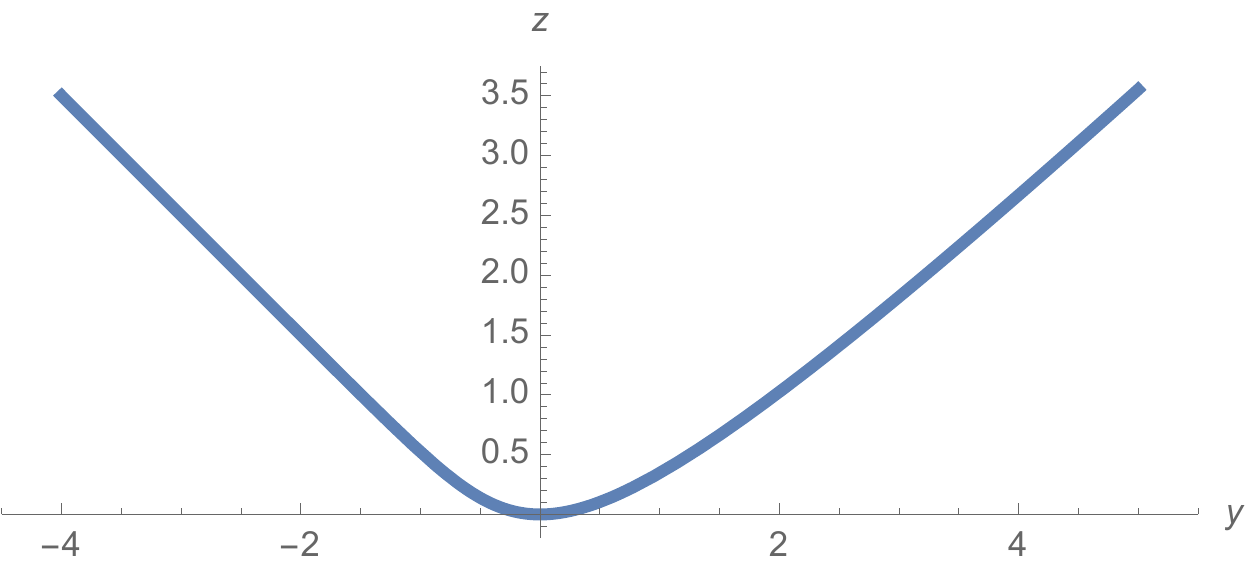}\quad \includegraphics[width=.3\textwidth]{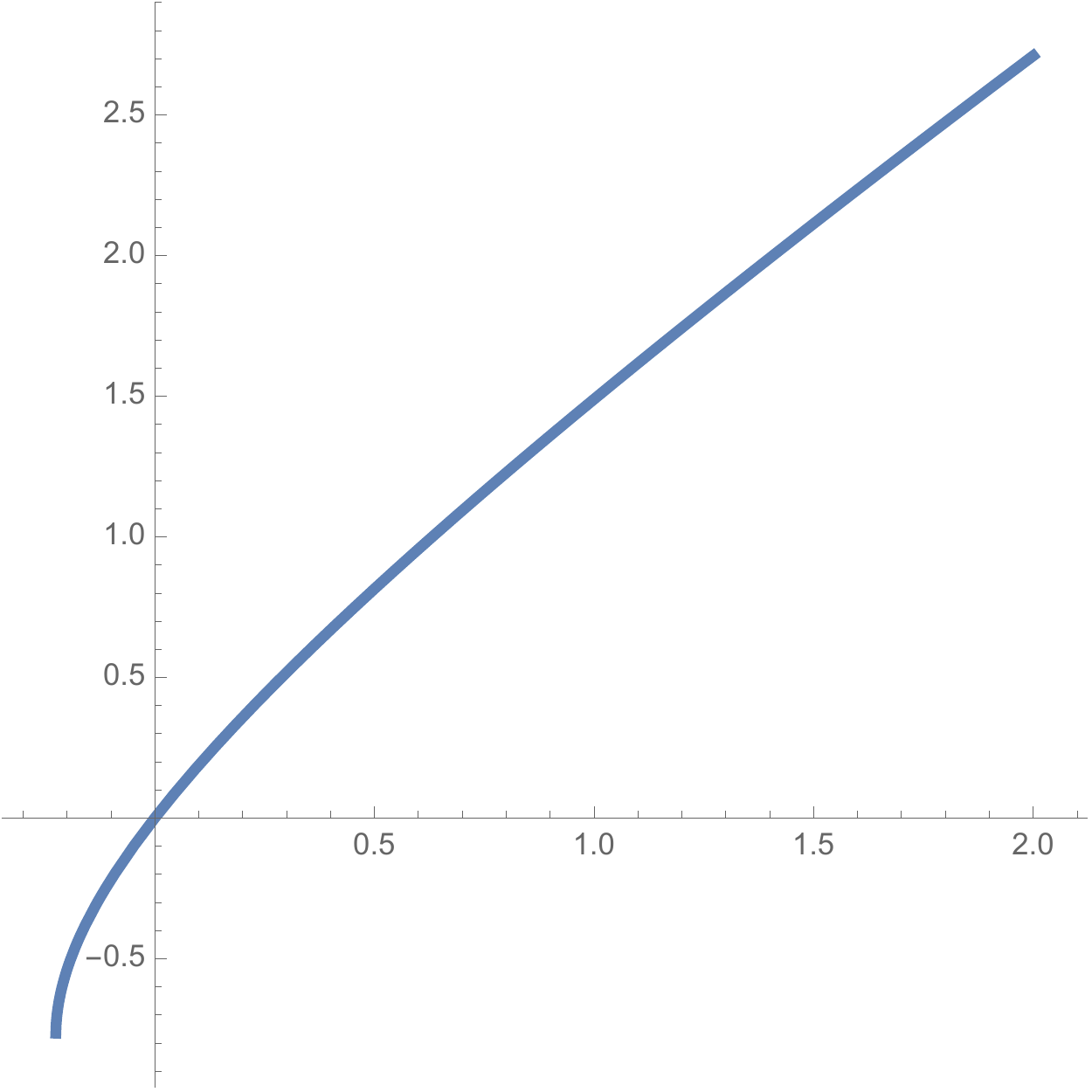}\end{center}
\caption{Numerical solutions of equations \eqref{gr0}. Case $1-u'^2>0$ (left) and $1-u'^2<0$ (right). }\label{fig2}
\end{figure}

%%%%%%%%%%%%%%%%%%%%%%%%%%%%%%%%%%

\section{Non-cylindrical translating solitons}\label{sec4}

In this section we investigate non-cylindrical ruled  translating solitons.  The parametrization of a ruled  surface is  $X(s,t)=\gamma (s)+tw(s)$ where $\gamma
:I\subset \mathbb{R}\rightarrow \mathbb{E}_{1}^{3}$, $\gamma =\gamma (s),$
is a regular curve and $w(s)$ is a nowhere
vanishing vector field along the curve $\gamma $.  For a description of the parametrizations of ruled surfaces in $\e_1^3$, we refer the reader to the reference \cite{dk}. If the surface is not cylindrical, then we can assume that $w'(s)\not=0$ in the interval $I$. In our investigation  we will separate the cases that rulings are or not lightlike.   First we consider the case that $w(s)$ is a lightlike vector field. 

\begin{theorem}\label{t2} If a ruled surface with lightlike rulings is a translating
soliton with respect to $\vec{v}\in \mathbb{E}_{1}^{3}$, then  it is cylindrical surface and the rulings are parallel to $\vec{v}$.
\end{theorem}

\begin{proof}
The proof is by contradiction. Suppose that the surface is not cylindrical,
in particular, $w'(s)\not=0$ in some subinterval of $I$, which we
will suppose that it is the very interval $I$. Since $w(s)$ is lightlike,
 $w'(s)$ is a spacelike vector field orthogonal to $w(s)$. Because the surface is
non-degenerate and $G=\langle w,w\rangle=0$, then $F=\langle \gamma',w\rangle \not=0$. We can take a new parameter $s$ such that $\langle w',w^{\prime
}\rangle =1$ and $\langle \gamma',w'\rangle =0$. Using $ 
G=0$ and $X_{tt}=0$, Equation \eqref{eq2} becomes 
\begin{equation*}
2(\gamma',w,w')=\langle \gamma',w\rangle \left(
(\gamma',w,\vec{v})+t(w',w,\vec{v})\right) .
\end{equation*} 
Because this is a polynomial identity on the variable $t$, we deduce 
\begin{equation*}
(w'(s),w(s),\vec{v})=0,
\end{equation*} 
for all $s\in I$. Then $\vec{v}$ is a linear combination of $w(s)$ and $w'(s)$ for all $s\in I$, hence there are  two smooth functions $a(s)$, $b(s)$   such that $ 
v=aw+bw'$. Since $\left\langle v,v\right\rangle =b^{2}$, we find that the function $b=b(s)$ is constant. Moreover, $b\not=0$ because
otherwise the vector field $w(s)$, which indicates the direction of the
rulings, would be parallel to $\vec{v}$. This proves that $\vec{v}$ is a
spacelike. Differentiating $\langle w',w'\rangle=1$ and $\langle w',w\rangle =0$, we have $\langle w'',w'\rangle =0$ and $ 
\langle w'',w\rangle +1=0$, respectively. If we now differentiate the identity $ 
v=a(s)w(s)+bw'(s)$, we have $0=a' w+aw' +bw''$. Multiplying by $w$, we conclude $0=b\langle w'',w\rangle $, a contradiction.
\end{proof}

%%%%%%%%%

Now consider the case that the rulings are not lightlike. If the surface $X(s,t)=\gamma (s)+tw(s)$ is not cylindrical, then we can suppose that $\langle w,w\rangle=\delta\in\{-1,1\}$. In particular,  $w'(s)$ is a vector field orthogonal to the rulings.  We separate in two cases depending if $w'(s)$ is lightlike or not.

\begin{theorem}\label{t3} Let $M$ be a ruled surface parametrized by  $X(s,t)=\gamma (s)+tw(s)$ where  $\langle w,w\rangle=\delta\in\{-1,1\}$. Assume that  $w'(s)$ is not lightlike. If $M$ is  a translating soliton with respect to $\vec{v}$, then  $M$ is a cylindrical surface or $M$ is a plane parallel to $\vec{v}$.
\end{theorem}

\begin{proof} Suppose that $M$ is not cylindrical. Then $w'$ does not vanish in some point, so  $w'(s)\not=0$ in some subinterval of $I$, which   we can suppose  $I$.   After a change of parameter, we can assume $\langle w',w'\rangle=\eta\in\{-1,1\}$. As in Euclidean case, we can replace $\gamma$ by the striction curve, which we denote by  $\gamma$ again. Under all these assumptions, we have
\begin{equation*}
\langle w,w\rangle =\delta,\quad \langle w' ,w' \rangle
=\eta,\quad \langle \gamma',w' \ \rangle =0,
\end{equation*} 
where $\delta,\eta\in \left\{ -1,1\right\}$. The translating soliton equation \eqref{eq2} becomes 
\begin{equation*}
\delta (  \gamma'+tw' , w,\gamma ''+tw'') -2\langle \gamma',w\rangle (w' ,\gamma' , w) =\epsilon \left( \delta\left( \langle\gamma',\gamma'\rangle+\eta t^2\right)  -\langle \gamma',w\rangle^2\right)( \gamma
' +tw' , w,\vec{v}).
\end{equation*}
This identity can be expressed as a polynomial equation on $t$ of type 
\begin{equation*}
\sum_{n=0}^{3}A_{n}(s) t^{n}=0.
\end{equation*}
Therefore all functions $A_{n}$ must vanish, $0\leq n\leq 3$. The
expressions of $A_2$ and $A_3$ are 
\begin{eqnarray*}
A_{2} &=&\delta (w'',w' ,w)-\epsilon\delta\eta
(\gamma', w ,\vec{v}), \\
A_{3} &=&\epsilon \delta\eta (w,w' ,\vec{v}).
\end{eqnarray*}
Since the surface is non-degenerate, $A_3=0$ yields  $(w' , w,\vec{v} 
)=0 $. This implies that $w(s) $ and $w'(s)$ are contained in a
plane $\Pi $ parallel to $\vec{v}$ for each $s\in I.$ Therefore the set $ 
\left\{ w(s) ,w'(s) ,w''(s)\right\} $ is linearly dependent for all $s\in I$.
From $A_{2}=0$ we have $( \gamma',w ,\vec{v}) =0$, hence $ 
\gamma'$ is also contained in $\Pi.$ This implies that $M$ is
part of the plane $\Pi$.
 \end{proof}

Finally, we consider the case that the rulings $w(s)$ are not lightlike with $\langle w,w\rangle =\delta\in\{-1,1\}$ and $w'(s)$ is lightlike for all $s\in I$. Because $w'(s)$ is a lightlike direction in the hyperboloid $\{p\in\e_1^3:\langle p,p\rangle=1\}$, then $w(s)$ is a straight line. Replacing the $s$-parameter, we assume that $w''=0$. In particular,  there are two vectors $\vec{a},\vec{b}\in \mathbb{E}_{1}^{3}$
such that $w(s)=\vec{a}s+\vec{b}$ and $\langle \vec{a},\vec{a}\rangle
=\langle \vec{a},\vec{b}\rangle =0$ and $\langle \vec{b},\vec{b}\rangle =1$.  After  a rigid motion of $\e_1^3$, we can consider 
$$w(s)=(1,s,s)=s(0,1,1)+(1,0,0), \quad  \vec{a}=(0,1,1),\,  \vec{b}=(1,0,0).$$
 With this choice, we have $w\times w'=-w'$. In the next theorem, we give a complete classification of these ruled surfaces that are translating solitons. The goal of theorem is precisely   that we show the explicit parametrization of the base curve $\gamma$.

\begin{theorem} \label{t4} Let $M$ be a non-cylindrical ruled surface parametrized by  $X(s,t)=\gamma (s)+t(1,s,s)$.   If $M$ is  a translating soliton, then $\langle\gamma',\vec{a}\rangle\not=0$ and $\vec{a}$ is orthogonal to   $\vec{v}$.  Furthermore, after dilations and translations, the base curve $\gamma$ is given by  
$$\gamma(s)=(x(s),z(s)+\Phi(s),z(s)),$$
according to the following cases:
\begin{enumerate}
\item Case $\vec{v}=(0,1,1)$. Then $\Phi(s)=\frac{1}{2\epsilon}\log(2\epsilon s+a)$, $a\in\r$, and $x(s)$ and $z(s)$ are given by \eqref{v0x} and \eqref{v0z}, respectively. 
\item Case $\vec{v}=(1,v_2,v_2)$. Depending on the integration constants, we have 
\begin{enumerate}
\item $\Phi(s)=\frac{1}{\epsilon(s-v_2)}$ and $x(s)$ and $z(s)$ are given by \eqref{a0x} and \eqref{a0z}, respectively. 
\item $\Phi(s)=-\frac{1}{ap}\arctan (p(s-v_2))$, where $a\in\r$, $a\not=0$,  and $p=\sqrt{\epsilon/a}$ if $\mbox{sgn}(\epsilon)=\mbox{sgn}(a)$. The functions   $x(s)$ and $z(s)$ are given by \eqref{a1x} and \eqref{a1z}; or  $\Phi(s)=-\frac{1}{2pa}\log\frac{1+p(s-v_2)}{1-p(s-v_2)}$, where $a\in\r$, $a\not=0$,  and  $p=\sqrt{-\epsilon/a}$ if $\mbox{sgn}(\epsilon)=-\mbox{sgn}(a)$. The functions  $x(s)$ and $z(s)$ are given by  \eqref{a2x} or \eqref{a2z},  respectively. 

\end{enumerate}
\end{enumerate}

\end{theorem}

\begin{proof}  The translation soliton equation \eqref{eq2} is  
\begin{equation*}
(\gamma'+tw' ,w,\gamma '')=\epsilon \left(
\langle \gamma',\gamma'\rangle +2t\langle \gamma
' ,w' \rangle \right) (\gamma'+tw' ,w,\vec{ 
v}).
\end{equation*} 
We write this identity as the polynomial equation 
\begin{equation*}
\sum_{n=0}^{2}B_{n}(s)t^{n}=0,
\end{equation*} 
where $B_{0}(s),$ $B_{1}(s),$ $B_{2}(s) $ are   given by 
\begin{eqnarray*}
B_{0} &=&(\gamma',w,\gamma '')-\epsilon \langle
\gamma',\gamma'\rangle (\gamma',w,\vec{v}), \\
B_{1} &=&(w' ,w,\gamma '')-\epsilon \left( \langle
\gamma',\gamma'\rangle (w' ,w,\vec{v})+2\langle
\gamma',w' \rangle (\gamma',w,\vec{v})\right) ,
\\
B_{2} &=&-2\epsilon \langle \gamma',w' \rangle
(w' ,w,\vec{v}).
\end{eqnarray*} 
From $B_2=0$, we have two cases. First, assume     $\langle \gamma',w' \rangle =0$  and we will see that this case is not possible. Otherwise, because $ 
\langle \gamma',w\rangle =0$, then $\langle \gamma',\vec{ 
a}\rangle =\langle \gamma',\vec{b}\rangle =0$. This implies that $ 
\gamma'$ is lightlike. Since $X_s=\gamma'+t w'$ and $X_t=w$,  the coefficients $E$ and $F$ of the first fundamental form are $0$, hence $EG-F^2=0$ and the surface would be
degenerated, a contradiction.

As a conclusion,   $\langle \gamma',w' \rangle \neq 0$. Then  $B_2=0$ is equivalent to $ 
(w' ,w,\vec{v})=0$. The identity $w\times w' =-w' $ yields 
$\left\langle w' ,\vec{v}\right\rangle =0$, proving that $\vec{a}$ and $\vec{v}$ are orthogonal to each other. We will assume that the base curve $\gamma  $ is not lightlike,  otherwise we replace $\gamma$ by $\gamma(s)+\lambda(s) w(s)$   for a
certain smooth function $\lambda=\lambda (s) $ to get that the new base curve is not lightlike. Let us introduce the notation
$$Q=\left\langle \gamma',w' \right\rangle,\quad R=\left\langle \gamma',\gamma ^{\prime
}\right\rangle.$$
Consider the orthogonal basis $\{\gamma',w, \gamma'\times
w\}$. Because $\langle\gamma',w\rangle=0$, then $\langle\gamma'\times w,\gamma'\times w\rangle=-R$. Thus
$$w'=\frac{Q}{R}\left( \gamma'+\gamma'\times
w\right).$$
Multiplying by $\vec{v}$, we obtain  
$$
\langle \gamma',\vec{v}\rangle =-(\gamma',w,\vec{v}).
$$
With this identity and   using $w\times w' =-w' $ again, the equations $B_0=B_1=0$ write now as  
\begin{equation}
(\gamma',w,\gamma '')=-\epsilon R\langle\gamma',\vec{v}\rangle,  \label{eq9}
\end{equation} 
\begin{equation}
\langle \gamma '',w' \rangle =-2\epsilon Q\langle
\gamma',\vec{v}\rangle,  \label{eq8}
\end{equation}
respectively. Definitively, the equations to solve are \eqref{eq9} and \eqref{eq8} with the extra condition that  $\langle \gamma',w\rangle=0$. Since $\langle \vec{v},\vec{a}\rangle=0$, then $v_2=v_3$ and   $\vec{v}=(v_1,v_2,v_2)$. Let $\gamma(s)=(x(s),y(s),z(s))$. The orthogonality condition $\langle \gamma',w\rangle=0$ writes as 
\begin{equation}\label{xx}
x'+s(y'-z')=0.
\end{equation} 
Thanks to \eqref{xx}, Equation \eqref{eq8} is $y''-z''=-2\epsilon(y'-z')^2(v_2-v_1 s)$ or equivalently,
\begin{equation}\label{yz}
y'(s)-z'(s)= \dfrac{1}{2\epsilon \int^s (v_2-  v_1 u)\ du}.
\end{equation}
Now \eqref{xx} implies
\begin{equation}\label{xxx}
x'(s)=- \dfrac{s}{2\epsilon \int^s (v_2-   v_1 u)\ du}.
\end{equation}
The explicit integration depends if $v_1$ is or not $0$. From now,  we will omit those integration constants that represent translations of the curve $\gamma$.
\begin{enumerate}
\item Case $v_1=0$. In particular $v_2\not=0$. After a dilation of the space, we can assume $\vec{v}=(0,1,1)$. From \eqref{yz} we deduce
$$y(s)=z(s)+\frac{1}{2\epsilon}\log(2\epsilon s+a),\quad a\in\r.$$
Then \eqref{xxx} implies
\begin{equation}\label{v0x}
x(s)=\frac{a \log (2 \epsilon s+a)}{4}-\frac{s}{2 \epsilon}.
\end{equation}
Now Equation  \eqref{eq9} is an ODE on $z$ given by 
$$-\epsilon(2\epsilon s+a)^2 z''-a\epsilon s+1-s^2=0.$$
The integration leads to
\begin{equation}\label{v0z}
z(s)=-\frac{\left(a^2+4\right) \log (2  \epsilon s+a )}{16 \epsilon }+b s-\frac{s^2 \epsilon }{8},\quad b\in\r.
\end{equation}

\item Case $v_1\not=0$. After a dilation, let $\vec{v}=(1,v_2,v_2)$. Now \eqref{yz} is
\begin{equation}\label{yz2}
y'(s)-z'(s)=-\frac{1}{\epsilon (s-v_2)^2+a},\quad a\in\r.
\end{equation}
The integration of the function $x(s)$ in \eqref{xxx} depends if $a$ is or not $0$.
\begin{enumerate}
\item Case $a=0$. Then 
\begin{equation}\label{a0x}
x(s)=\frac{1}{\epsilon}\left(\log(s-v_2)-\frac{v_2}{s-v_2}\right)
\end{equation}
and 
$$y(s)=z(s)+\frac{1}{\epsilon (s-v_2)}.$$
Now \eqref{eq9} is 
$$1+sv_2+\epsilon (s-v_2)^3z''=0,$$
obtaining
\begin{equation}\label{a0z}
z(s)=\frac{1}{2\epsilon}\left(2v_2\log(s-v_2)-\frac{1+v_2^2}{s-v_2}\right)+bs,\quad b\in\r.
\end{equation}
\item Case $a\not=0$. The integration of \eqref{yz2} depends on the sign of $a$.
\begin{enumerate}
\item Case $\mbox{sgn}(a)=\mbox{sgn}(\epsilon)$. Let $p=\sqrt{\epsilon/a}$ and $\phi(s)=p(s-v_2)$. Then 
$$y(s)=z(s)-\frac{1}{ ap}\arctan{\phi(s)}.$$
The integration of \eqref{xxx} yields
\begin{equation}\label{a1x}
x(s)=\frac{1}{2\epsilon}\log(1+\phi(s)^2)+\frac{ v_2}{a p} \arctan{\phi(s)}.
\end{equation}
Then \eqref{eq9} is
$$-\epsilon(1+v_2 s)(s-v_2)+as-((s-v_2)^2+\epsilon a)^2z''=0.$$
The solution of this equation is
\begin{equation}\label{a1z}
z(s)=\frac{\left(p^2 \left(v_2^2+1\right)-1\right) \arctan{\phi(s)}+p v_2 \log \left(1+\phi(s)^2\right)}{2 \epsilon }+b s, \quad b\in\r.
\end{equation}

\item Case $\mbox{sgn}(a)=-\mbox{sgn}(\epsilon)$. Following the same arguments as in the foregoing case,    if $p=\sqrt{-\epsilon/a}$ and  $\phi(s)=p(s-v_2)$, we obtain 
$$y(s)=z(s)-\frac{1}{2pa}\log\frac{1+\phi(s)}{1-\phi(s)}.$$
\begin{equation}\label{a2x}
x(s)=\frac{1}{2\epsilon}\log(1-\phi(s)^2)+\frac{v_2}{2ap}\log\frac{1+\phi(s)}{1-\phi(s)}.
\end{equation}
\begin{equation}\label{a2z}
z(s)=\frac{\epsilon v_2}{2}   \log \left(1-\phi(s)^2\right)+\frac{1-a \epsilon +v_2^2}{2ap}  \mbox{arctanh}{\phi(s)}+b s,\quad b\in\r.
\end{equation}

\end{enumerate}
\end{enumerate} 
\end{enumerate}
\end{proof}

We conclude this section with the next consequence of Theorems \ref{t2}, \ref{t3} and \ref{t4}.  

\begin{corollary} Let $M$ be a (non-planar) translating soliton in $\e_1^3$ with respect to $\vec{v}$. Assume that $M$ is a non-cylindrical ruled surface parametrized by $X(s,t)=\gamma(s)+t w(s)$. Then, and after a reparametrization of the surface, $w=w(s)$ is a lightlike straight line and $\langle w',\vec{v}\rangle=0$. 
\end{corollary}
%%%%%%%%%%%%%%%%%%%%%%%%%%%%

\section*{Acknowledgment} Rafael L\'opez has been partially supported by the grant no. MTM2017-89677-P, MINECO/AEI/FEDER,
215 UE.


\begin{thebibliography}{9} 

\bibitem{al} M. E. Aydin, R. L\'opez, Translating solitons by separation of
variables, preprint, 2021.

\bibitem{bf}  M.~Barros, A.~Ferr\'andez,  Null scrolls as solutions of a sigma model, Journal of Physics A: Mathematical and Theoretical, 45, (2012), 145203-1/145203-12.   

\bibitem{I}   I. Castro, I. Castro-Infantes, Curves in Lorentz-Minkowski plane: elasticae, catenaries and
grim-reapers, Open Math. 16 (2018), 747--766.

 

\bibitem{dk} F. Dillen, W. K\"{u}hnel, Ruled Weingarten surfaces in Minkowski
3-space, Manuscripta Math. 98 (1999), 307 -- 320.

\bibitem{ec1} K. Ecker,   On mean curvature flow of spacelike hypersurfaces in asymptotically flat spacetimes, J. Austral. Math. Soc. Ser. A, 55 (1993) 41--59.

\bibitem{ec2} K. Ecker, Interior estimates and longtime solutions for mean curvature flow of noncompact spacelike hypersurfaces in Minkowski space, J. Differ. Geom. 46   (1997), 481--498.
\bibitem{ec3} K. Ecker, G. Huisken, Parabolic methods for the construction of spacelike slices of prescribed mean curvature in cosmological spacetimes, Comm. Math. Phys. 135 (1991) 595--613.

\bibitem{ge} C. Gerhardt, Hypersurfaces of prescribed mean curvature in Lorentzian manifolds,  Math. Z. 235 (2000), 83--97.

\bibitem{Hieu} D.T. Hieu, N. M. Hoang, Ruled minimal surfaces in $\mathbb{R}%
^{3}$ with density $e^{z}$, Pacific J. Math. 243 (2009) no. 2, 277-285.

\bibitem{hsi1} G. Huisken, C. Sinestrari, Mean curvature flow singularities
for mean convex surfaces, Calc. Var. 8 (1999) 1--14.

\bibitem{il} T. Ilmanen, Elliptic regularization and partial regularity for
motion by mean curvature, Mem. Amer. Math. Soc. 108, x+90 (1994)
 
 \bibitem{ll} B. Lambert, J. D. Lotay,  
Spacelike mean curvature flow, J. Geom. Anal. 31 (2021),   1291--1359. 

\bibitem{lo} R. L\'opez,  Differential geometry of curves and surfaces in Lorentz-Minkowski space, Int. Electron. J. Geom. 7 (2014), 44--107.

\bibitem{we} T. Weinstein,   An Introduction to Lorentz Surfaces, De Gruyter Expositions in Mathematics, vol. 22. Walter de Gruyter  Co., Berlin (1996).
\end{thebibliography}
\end{document}